\documentclass[12pt, a4paper]{article}
	
	\usepackage[utf8]{inputenc}          
	\usepackage{amsmath, amssymb, amsfonts} 
	\usepackage{graphicx}                
	\usepackage{geometry}                
	\usepackage{anyfontsize} %
	\usepackage{float}
	\usepackage{authblk}
	\usepackage{colortbl}   
	
	\usepackage{xcolor}     
	
	\usepackage{natbib}
	
	\usepackage{steinmetz}

	\newcommand{\peme}{p_{\text{em}}}

\usepackage{algorithm}

\usepackage{amsthm} 
\newtheorem{theorem}{Theorem}
\newtheorem{lemma}{Lemma} 

\newtheorem{proposition}{Proposition}
\newtheorem{remark}{Remark}

\begin{document}

		\title{Local Search for Almost-Spanning Square Grids in
			Erd\H{o}s--Rényi Random Graphs}
		\author[1]{Dávid Ferenczi\thanks{Corresponding author:
				\texttt{david.ferenczi@maastrichtuniversity.nl}}}
		\author[1]{Alexander Grigoriev}
		
		\affil[1]{Department of Data Analytics and Digitalization,
			School of Business and Economics,
			Maastricht University,
			Maastricht, The Netherlands}
		
		\date{}
		\maketitle

				\maketitle
		
	\begin{abstract}Finding large lattice subgraphs in sparse Erd\H{o}s--Rényi random graphs is a classical problem at the interface of random graph theory and algorithms. General bounded-degree embedding and universality theorems give powerful results for broad graph families, but when specialized to square grids they operate at densities substantially larger than the grid-emergence scale. In this paper we exploit the specific geometry of the square grid. We introduce the Quarantined Local Search, a three-phase local algorithm that separates the construction of an initial boundary from the later corner-closure process and controls adaptive negative exposure through bounded pair-test histories. We prove that, for every fixed $\delta\in (0,1)$, there exists $C_\delta>0$ such that the algorithm embeds a $k\times k$ square grid with $k^2\le (1-\delta)n$ in $G(n,p)$ with high probability whenever $p\ge C_\delta\sqrt{\ln k/n}$. Thus, for $k^2=\Theta(n)$, a density of order $\sqrt{\log n/n}$ is sufficient, a factor of order $\sqrt{\log n}$ above the corresponding $n^{-1/2}$ emergence scale.\end{abstract}


		%
		


		\section{Introduction}\label{sec:int}
		
		Given a graph $H$ and an Erd\H{o}s--Rényi random graph $G\sim G(n,p)$, a classical question in random graph theory asks for the threshold probability $\peme$ such that, if $p>\peme$, then $G$ contains a copy of $H$ as a subgraph with high probability. Throughout the paper, subgraphs are not necessarily induced, i.e., extra edges of the host graph among subgraph vertices may be present. For a graph $G$, write $\rho(G)=e(G)/v(G)$. We call $H$ balanced if $\rho(H')\le\rho(H)$ for every nonempty subgraph $H'\subseteq H$.

		For fixed balanced graphs, the classical threshold result \cite{bollobas_1981} gives that the threshold for the appearance of $H$ has order
		\begin{equation*}
			n^{-v(H)/e(H)}.
		\end{equation*}

While the fixed-graph threshold is classical, \citet{alon_1992} posed the corresponding question for large spanning graphs, where the target subgraph $H$ contains every vertex of the host graph, i.e. $V(H)=V(G)$. They formulated the problem in a general way and proved bounds for the $d$-cube and the square grid. These bounds were later sharpened by \citet{riordan_2000}, who proved general emergence results for spanning structures in random graphs. In contrast with strictly spanning embeddings, we consider the almost-spanning setting in which the target is allowed to contain up to
a fixed fraction $1-\delta$ of the host vertices.

The closest grid-specific constructive result is due to
	\citet{fernandez_manoussakis}. They gave a greedy stepwise construction
	showing that, for fixed $\varepsilon\in(0,1)$ and
	\begin{equation*}
		p=
		\left(
		\frac{(1+\varepsilon)\log n}{n}
		\right)^{1/2},
	\end{equation*}
	$G(n,p)$ contains with high probability a square grid on at least
	$\varepsilon n/2$ vertices. Thus their construction already operates
	at the same $\sqrt{\log n/n}$ density order as the present algorithm
	in the macroscopic regime. The distinction in the present result is
	the occupancy range: for every fixed $\delta>0$, our construction
	allows grids with up to $(1-\delta)n$ vertices.

A different line of work studies embedding and universality for broad families of bounded-degree graphs. \citet{ferber_2017} proved an almost-spanning embedding theorem for arbitrary graphs on $(1-\varepsilon)n$ vertices with maximum degree at most $\Delta$, for $\Delta\ge 5$, at densities
\begin{equation*}
p\gg \left(n^{-1}\log^{1/\Delta}n\right)^{2/(\Delta+1)}.
\end{equation*}
Subsequently, \citet{ferber_nenadov_2018} proved a spanning universality theorem for all $n$-vertex graphs of maximum degree at most $\Delta\ge 3$, at densities
\begin{equation*}
p\ge \left(n^{-1}\log^3 n\right)^{1/(\Delta-1/2)}.
\end{equation*}
These results are substantially more general than the present paper. However, when specialized to square grids, they apply at densities much larger than the grid-emergence scale. Our focus is therefore different: we restrict to one geometric family and exploit its nested lattice structure to obtain a local construction when $p\geq C_\delta \sqrt{\frac{\log k}{n}}$.

The present paper sits between these two perspectives. Compared with the existential square-grid results of \citet{riordan_2000}, our result is algorithmic but, in the regime $k^2=\Theta(n)$, pays a factor of order $\sqrt{\log n}$ in the required edge density. Compared with the general bounded-degree embedding and universality results of \citet{ferber_2017,ferber_nenadov_2018}, our result is much narrower, applying only to square grids, but reaches a substantially sparser density for this specific target family. The main tool is a three-phase algorithm, the \emph{Quarantined Local Search}, which separates the construction of the initial boundary from the later corner-closure process. The quarantine mechanism prevents one-edge exposure history 	from contaminating the active frontier, while a bounded partner-history invariant controls the remaining pair-test dependencies during the grid-filling phase.	The comparison with the most relevant previous results is summarized in Table~\ref{tab:comparison}.

\begin{table}[htbp]
	\centering
	\footnotesize
	\caption{Comparison with representative previous embedding results
		relevant to square grids.}
	\label{tab:comparison}
	\begin{tabular}{
			p{0.18\linewidth}
			p{0.18\linewidth}
			p{0.27\linewidth}
			p{0.18\linewidth}
			p{0.13\linewidth}}
		\hline
		Result
		& Target
		& Density guarantee
		& Type of result
		& Occupancy
		\\
		\hline
		
		\citet{fernandez_manoussakis}
		& Square grids
		& $\displaystyle
		p=
		\left(
		\frac{(1+\varepsilon)\log n}{n}
		\right)^{1/2}$
		& Greedy constructive embedding
		& At least $\varepsilon n/2$
		\\[1mm]
		
		\citet{riordan_2000}
		& Square grid $L_k$
		& $\displaystyle
		p=\omega(n)n^{-1/2}$;
		threshold scale $n^{-1/2}$
		& Existential; second moment
		& Spanning
		\\[1mm]
		
		\citet{ferber_2017}
		& Bounded-degree graphs,
		$\Delta\ge5$
		& $\displaystyle
		p\gg
		\left(
		n^{-1}\log^{1/\Delta}n
		\right)^{2/(\Delta+1)}$
		& General embedding theorem
		& $(1-\varepsilon)n$
		\\[1mm]
		
		\citet{ferber_nenadov_2018}
		& Bounded-degree graphs,
		$\Delta\ge3$
		& $\displaystyle
		p\ge
		\left(
		n^{-1}\log^3n
		\right)^{1/(\Delta-1/2)}$
		& Spanning universality
		& Spanning
		\\[1mm]
		
		This paper
		& Square grids
		& $\displaystyle
		p\ge
		C_\delta
		\sqrt{\frac{\log k}{n}}$
		& Randomized local search
		& $k^2\le(1-\delta)n$
		\\
		\hline
	\end{tabular}
\end{table}
For comparison with square grids, the result of \citet{ferber_2017} may be applied with $\Delta=5$, giving
\begin{equation*}
p\gg n^{-1/3}(\log n)^{1/15},
\end{equation*}
whereas \citet{ferber_nenadov_2018} may be applied with $\Delta=4$, giving
\begin{equation*}
p\ge n^{-2/7}(\log n)^{6/7}.
\end{equation*}
Thus both general bounded-degree results require substantially larger
edge probabilities than the square-grid-specific bound obtained here.

Table~\ref{tab:comparison} summarizes the position of the present
result. Relative to \citet{fernandez_manoussakis}, we retain the same
asymptotic density order for macroscopic square grids while extending
the constructive guarantee from a fixed linear fraction to
every fixed almost-spanning regime $k^2\le(1-\delta)n$. Relative to
\citet{riordan_2000}, our result is constructive and local, but in the
regime $k^2=\Theta(n)$ pays a factor of order $\sqrt{\log n}$ over the
existential threshold scale. Finally, compared with the general
bounded-degree embedding and universality results of
\citet{ferber_2017,ferber_nenadov_2018}, our result is substantially
narrower but reaches a sparser density for square grids by exploiting
their nested lattice geometry.

		The main result of our paper goes as follows:

		\begin{theorem}\label{thm:main}
			Fix $\delta\in(0,1)$, and let $k=k(n)\to\infty$ satisfy
			\begin{equation*}
				k^2\le(1-\delta)n.
			\end{equation*}
			There exists a constant $C_\delta>0$ and a randomized local-search
			algorithm such that, whenever
			\begin{equation*}
				p\ge C_\delta\sqrt{\frac{\ln k}{n}},
			\end{equation*}
			the algorithm finds a subgraph of $G\sim G(n,p)$ isomorphic to the
			$k\times k$ square grid with probability $1-o(1)$.
		\end{theorem}
When $k^2=\Theta(n)$, we have $\log k=\Theta(\log n)$. Consequently, Theorem~\ref{thm:main} shows that a density of order $\sqrt{\frac{\log n}{n}}$ is sufficient for the algorithm to succeed with high probability. Thus, in this regime, the algorithm incurs a factor of order $\sqrt{\log n}$ over the existential square-grid emergence scale.
		
		The remainder of the paper goes as follows: in Section \ref{sec:Prelim} we present our main algorithm, that attempts to build a grid of desired size. In Section \ref{sec:Proof} we prove Theorem~\ref{thm:main}. In Section \ref{sec:Discussion} we discuss the limitations and possible extensions of our proposed method.

		\section{Algorithmic Overview}\label{sec:Prelim}

In this section we present the local search algorithm that takes an Erd\H{o}s--Rényi graph $G\sim G(n,p)$, the model parameter $p$, and a target dimension $k$ as input, and attempts to construct a $k\times k$ square grid. We start by describing the two basic vertex embedding methods we rely on.

		\begin{itemize}
			
			\item \textbf{Linear Extension:} We begin by constructing an initial path of $2k-1$ vertices. Starting from an arbitrary root vertex, the algorithm repeatedly searches for a candidate adjacent to the current path tip. If a queried candidate shares an edge with the current tip of the path, it is permanently embedded to extend the sequence. This expansion repeats until the path reaches the required length of $2k-1$, serving as the L-shaped boundary for the upcoming grid expansion.
			
			\item \textbf{Corner Closure:} Once the initial path of $2k-1$ vertices is established, we map it to the outer boundary of our target grid, forming an L-shaped boundary. The algorithm then constructs the remainder of the $k \times k$ grid by iteratively adding successively smaller, nested L-shapes parallel to this initial boundary. The filling begins at the corner of the embedded boundary, say, vertex $v_k$. The algorithm queries the unmapped reservoir for a vertex that connects to both $v_{k-1}$ and $v_{k+1}$, effectively filling the inward-facing pocket to complete the first unit square. This operation creates a new active boundary. By systematically closing the adjacent corners along this frontier, the algorithm constructs a secondary, inner L-shape consisting of $2k-3$ vertices. This process repeats, adding nested L-shapes of decreasing size ($2k-5, 2k-7, \dots$) and sweeping diagonally inward until the final opposite corner is closed and the grid is completely filled.
			
		\end{itemize}

\subsection{Query model}

The algorithm is analyzed in a local-observation model. A Linear Extension query tests whether a candidate vertex is adjacent to the current path tip. These one-edge queries are used only during Phase 1, and any vertices with such exposure history are quarantined before the first Corner Closure layer is constructed.

A Corner Closure query is a pair-query: for a candidate $x$ and an active closure pair $(u,v)$, the query returns SUCCESS if both $xu\in E(G)$ and $xv\in E(G)$, and returns FAILURE otherwise. On FAILURE, the algorithm records only that $x$ was not adjacent to both $u$ and $v$; it does not reveal or retain which of the two edges was absent. Thus the Corner Closure filtrations used below are generated by successful embeddings and by these pair-test outcomes, not by individual missing-edge exposures.

During a Corner Closure, each candidate is tested at most once against the current closure pair. A candidate that fails remains available for subsequent, different closure pairs. Candidates are inspected in a fixed order, or more generally in an order measurable with respect to the revealed filtration.
Although $G$ may be represented explicitly, the algorithm is analyzed as an information-restricted local search procedure; after a failed pair-query it discards the identity of the missing edge and retains only the Boolean failure outcome.
\subsection{Structural Lemmas}
We first record the deterministic fact that the target grid can indeed be completed by the prescribed nested-L sequence of Corner Closures. We then state the probabilistic lemmas used to bound the failure probability of these operations.

\begin{lemma}[Nested-L construction]\label{lem:nested_L}
	Let $S_k$ be the $k\times k$ square grid. After embedding an $L$-shaped boundary path of $2k-1$ vertices, $S_k$ can be completed by a deterministic sequence of Corner Closures. The successive layers add
	\begin{equation*}
	2k-3,\ 2k-5,\ \ldots,\ 1
	\end{equation*}
	vertices, and hence the total number of embedded vertices is
	\begin{equation*}
	(2k-1)+(2k-3)+\cdots+1=k^2.
	\end{equation*}
\end{lemma}

\begin{proof}
	Identify the vertices of $S_k$ with pairs $(i,j)\in [k]\times [k]$, with grid edges between vertices at Manhattan distance one. Embed the initial $L$-shaped boundary as
	\begin{equation*}
	\{(1,j):1\le j\le k\}\cup \{(i,1):2\le i\le k\}.
	\end{equation*}
	For each $r=2,\ldots,k$, define the $r$-th nested layer by
	\begin{equation*}
	L_r=\{(r,j):r\le j\le k\}\cup \{(i,r):r+1\le i\le k\}.
	\end{equation*}
	This layer has
	\begin{equation*}
	(k-r+1)+(k-r)=2(k-r)+1
	\end{equation*}
	vertices.
	
	We embed $L_r$ by first adding $(r,r)$, which is adjacent to the already embedded vertices $(r-1,r)$ and $(r,r-1)$. We then add the remaining vertices of the horizontal part from left to right; each $(r,j)$, $j>r$, is adjacent to $(r-1,j)$ and $(r,j-1)$, both of which have already been embedded. Similarly, we add the vertical part from top to bottom; each $(i,r)$, $i>r$, is adjacent to $(i,r-1)$ and $(i-1,r)$, both already embedded at that moment.
	
	Thus every added vertex is introduced by a Corner Closure using two previously embedded grid neighbours. Summing the layer sizes gives
	\begin{equation*}
	(2k-1)+\sum_{r=2}^k \bigl(2(k-r)+1\bigr)
	=(2k-1)+(2k-3)+\cdots+1=k^2.
	\end{equation*}
\end{proof}

In particular, the prescribed nested-$L$ schedule never uses the same closure pair twice.

\begin{lemma}\label{lem:linear_ext}
	Let $G\sim G(n,p)$, and suppose a Linear Extension step is performed using a candidate pool of size $N$. While building a path of length $2k-1$, the probability $\mathbb P(F_{\textup{linear}})$ that a single Linear Extension step fails to find a suitable candidate is bounded above by
	\begin{equation*}
	\mathbb{P}(F_{\textup{linear}})\le (1-p)^{N-2k+2}.
	\end{equation*}
\end{lemma}

\begin{proof}
	At any given step during the Linear Extension, the algorithm attempts to attach a new vertex to the current tip of the path. The target path length is $2k-1$. In the worst case, the algorithm is searching for the final vertex of the path, meaning that $2k-2$ vertices of the candidate pool have already been embedded. Hence at least
	\begin{equation*}
	N-(2k-2)=N-2k+2
	\end{equation*}
	candidate vertices remain available.
	
	Each remaining candidate is adjacent to the current path tip independently with probability $p$. Therefore, the extension step fails only if all remaining candidates fail to connect to the tip, which has probability at most
	\begin{equation*}
	(1-p)^{N-2k+2}.
	\end{equation*}
\end{proof}

		In the proof above we used independence of one-edge queries. For Corner Closure operations, the corresponding issue is more delicate. If a candidate vertex has previously failed a pair-query, then the algorithm has learned that at least one of two tested edges was absent, although it has deliberately not retained which edge failed. Therefore, the conditional success probability of a later pair-query need not be exactly $p^2$. The next lemmas show that, under the nested-L construction, the amount of relevant negative pair-test history remains uniformly bounded.

		Before estimating the conditional success probability of a \textbf{Corner Closure}, we must bound the amount of pair-test history that can involve any active boundary vertex. This prevents the adaptive search process from generating an unbounded collection of negative constraints on candidates that are returned to the reservoir after failed tests.

\begin{lemma}[Bounded Partner Invariant]\label{lem:partner_bound}
	During the iterative construction of the grid via \textbf{Corner Closures}, each vertex $u$ of the partially embedded target grid occurs in at most four distinct closure pairs $(u,w)$ over the entire execution of the algorithm. In particular, while $u$ belongs to the active frontier, the set of distinct vertices $w$ for which $(u,w)$ is queried as a closure pair has cardinality at most $4$.
\end{lemma}

\begin{proof}
	By definition, a \textbf{Corner Closure} completes one previously unfilled unit square of the target $k\times k$ grid. If $u$ occurs in a closure pair $(u,w)$, then $u$ and $w$ are two already embedded vertices of the corresponding unit square, and the closure operation adds the required new grid vertex.
	
	A fixed vertex $u$ in a square grid is incident with at most four unit squares. Each unit square is completed at most once during the prescribed nested-$L$ construction. Moreover, whenever a square incident with $u$ is completed, it produces at most one closure pair containing $u$. Hence each of the at most four unit squares incident with $u$ contributes at most one distinct partner $w$.
	
	Therefore, under the prescribed nested-$L$ construction, $u$ belongs to at most four distinct closure pairs during the whole construction.
\end{proof}

\begin{lemma}\label{lem:corner_closure}
	Let $G\sim G(n,p)$ be an Erd\H{o}s--Rényi graph, and suppose a Corner Closure step is performed with active closure pair $(u,v)$ and active reservoir of size $N$. Under the pair-query model, for any realizable Corner Closure history $\mathcal F$ before this step, the conditional probability $\mathbb{P}(F_{\textup{corner}}\mid \mathcal F)$ that the algorithm fails during this step is bounded above by
	\begin{equation*}
		\mathbb{P}(F_{\text{corner}}\mid \mathcal F)
		\le
		\left(1-p^2(1-p)^8\right)^N.
	\end{equation*}
\end{lemma}

\begin{proof}
Throughout this proof, all conditional probabilities are understood on realizable history branches of positive probability. In particular, all conditioning events appearing below have positive probability.
For a candidate $x_i$ remaining in the active reservoir, every previous \textbf{Corner Closure} test involving $x_i$ has failed, since a successful candidate is immediately embedded and removed from the reservoir. 
Let $A_i$ be the event that $x_i$ is adjacent to both vertices of the current closure pair:
\begin{equation*}
	A_i=\{x_i u \in E(G),\ x_i v \in E(G)\}.
\end{equation*}

For every embedded grid vertex $z$, write
\begin{equation*}
X_z=\mathbf{1}_{\{x_i z\in E(G)\}}.
\end{equation*}
Thus $A_i=\{X_u=X_v=1\}$.

For the fixed revealed history prior to the current test, let $Q_{x_i}$ be
the auxiliary graph whose vertices are the already embedded grid vertices, and edges are the closure pairs against which $x_i$
was previously tested. Since $x_i$ remains in the reservoir, every such
test returned FAILURE. The corresponding event is
\begin{equation*}
\mathcal H_i
=
\bigcap_{ab\in E(Q_{x_i})}
\{X_aX_b=0\}.
\end{equation*}

We continue by defining the exposed neighbourhood to the boundary that is being augmented into a square, with the help of the auxiliary graph $Q_{x_i}$, i.e. a set of vertices that previously were used to connect either to $u$ or $v$.  Let
\begin{equation*}
W=N_{Q_{x_i}}(u)\cup N_{Q_{x_i}}(v).
\end{equation*}
By Lemma~\ref{lem:partner_bound},
\begin{equation*}
|W|
\le d_{Q_{x_i}}(u)+d_{Q_{x_i}}(v)
\le 8.
\end{equation*}
Moreover, the current pair $uv$ is not an edge of $Q_{x_i}$, since
$x_i$ has not previously been tested against the current closure pair. Consequently,
\begin{equation*}
W\cap\{u,v\}=\varnothing.
\end{equation*}

With the help of the set $W$ corresponding to the exposed neighbourhood we can split the history in $\mathcal H_i$, into local and residual edges. Let
\begin{equation*}
R=V(Q_{x_i})\setminus\bigl(\{u,v\}\cup W\bigr),
\end{equation*}
and partition the edges of $Q_{x_i}$ into the corresponding residual and local sets, as
\begin{equation*}
E_R=\{ab\in E(Q_{x_i}):a,b\in R\},
\qquad
E_{\mathrm{loc}}=E(Q_{x_i})\setminus E_R.
\end{equation*}
Define
\begin{equation*}
\mathcal H_R
=
\bigcap_{ab\in E_R}\{X_aX_b=0\},
\qquad
\mathcal H_{\mathrm{loc}}
=
\bigcap_{ab\in E_{\mathrm{loc}}}\{X_aX_b=0\}.
\end{equation*}
Then, by construction,
\begin{equation*}
\mathcal H_i=\mathcal H_R\cap\mathcal H_{\mathrm{loc}},
\end{equation*}
and in particular
\begin{equation*}
\mathcal H_i\subseteq\mathcal H_R.
\end{equation*}

We continue by introducing an event $B$ which is easy to work with in terms of revealed edge history. Let
\begin{equation*}
B=\{X_w=0\text{ for every }w\in W\}.
\end{equation*}

Next we show that every edge in $E_{\mathrm{loc}}$ has at least one endpoint in $W$.
Indeed, an edge in $E_{\mathrm{loc}}$ is not contained entirely in $R$,
so one of its endpoints lies in $\{u,v\}\cup W$. If that endpoint is
already in $W$, there is nothing to prove. If it is $u$ or $v$, then
its other endpoint belongs to $W$ by the definition of $W$, since
$uv\notin E(Q_{x_i})$.

Consequently,
\begin{equation*}
B\subseteq\mathcal H_{\mathrm{loc}},
\end{equation*}
and therefore
\begin{equation*}
B\cap\mathcal H_R
\subseteq
\mathcal H_{\mathrm{loc}}\cap\mathcal H_R
=
\mathcal H_i.
\end{equation*}
Hence
\begin{equation*}
A_i\cap B\cap\mathcal H_R
\subseteq
A_i\cap\mathcal H_i.
\end{equation*}

Using these inclusions we can continue by noting that, the events $A_i$, $B$, and $\mathcal H_R$ depend respectively on the
disjoint collections of edge indicators
\begin{equation*}
\{X_u,X_v\},\qquad
\{X_w:w\in W\},\qquad
\{X_r:r\in R\}.
\end{equation*}
Therefore, by independence of the Erd\H{o}s--Rényi edge indicators,
\begin{equation*}
\mathbb P(A_i\cap B\cap\mathcal H_R)
=
p^2(1-p)^{|W|}\mathbb P(\mathcal H_R).
\end{equation*}

Using the preceding inclusion,
\begin{equation*}
\mathbb P(A_i\cap\mathcal H_i)
\ge
p^2(1-p)^{|W|}\mathbb P(\mathcal H_R).
\end{equation*}
Since $\mathcal H_i\subseteq\mathcal H_R$,
\begin{align*}
	\mathbb P(A_i\mid\mathcal H_i)=
	\frac{\mathbb P(A_i\cap\mathcal H_i)}
	{\mathbb P(\mathcal H_i)} \ge
	p^2(1-p)^{|W|}
	\frac{\mathbb P(\mathcal H_R)}
	{\mathbb P(\mathcal H_i)} \ge
	p^2(1-p)^{|W|}
	\ge
	p^2(1-p)^8.
\end{align*}

Let $\mathcal F_{i-1}$ denote the complete revealed Corner Closure
history immediately before $x_i$ is tested. On any fixed realizable
history branch, the information in $\mathcal F_{i-1}$ involving
$x_i$ is precisely the collection of failure constraints encoded by
$\mathcal H_i$. All remaining revealed information depends on host-edge
variables distinct from the $x_i$-star variables appearing above.
Hence, by edge independence and deferred decisions,
\begin{equation*}
\mathbb P(A_i\mid\mathcal F_{i-1})
\ge p^2(1-p)^8.
\end{equation*}

Set
\begin{equation*}
	q=p^2(1-p)^8.
\end{equation*}
At every sequential candidate test, conditional on the complete revealed
history up to that test, the candidate fails with probability at most
$1-q$. A Corner Closure runs out of suitable vertices only if all $N$ candidates in the active
reservoir fail. Therefore, by repeated conditioning and the conditional
chain rule,
\begin{equation*}
	\mathbb P(F_{\mathrm{corner}}\mid\mathcal F)
	\le (1-q)^N
	=
	\left(1-p^2(1-p)^8\right)^N.
\end{equation*}
			
\end{proof}

		\begin{remark}[Thinning]
	The lower bound in Lemma~\ref{lem:corner_closure} contains the factor
	$(1-p)^8$. This factor arises from the conservative witness event used
	to control previous failed pair-queries, and therefore deteriorates when
	$p$ is close to $1$, even though increasing the density of the host
	graph should not make the embedding problem harder. Phase 0 removes this
	artifact. We fix $p_0\in(0,1)$ with
	\begin{equation*}
	(1-p_0)^8\ge\frac12,
	\end{equation*}
	and, whenever the original density exceeds $p_0$, work instead in an
	independently thinned subgraph with law $G(n,p_0)$. Since this
	operational graph is a subgraph of the original host graph, success after
	thinning implies success in the original graph. Consequently, throughout
	the probabilistic analysis we may assume
	\begin{equation*}
	p\le p_0
	\qquad\text{and hence}\qquad
	(1-p)^8\ge\frac12.
	\end{equation*}
\end{remark}

Lemma~\ref{lem:corner_closure} applies to Corner Closure steps whose candidate histories are generated only by pair-tests relative to the active frontier. The vertices queried during the initial Linear Extension may carry individual one-edge exposure history relative to the initial path $P_{2k-1}$. To prevent this history from entering the first Corner Closure layer, the algorithm temporarily quarantines all unused vertices from the initial pool. After the first interior layer is completed, the active frontier no longer contains vertices of $P_{2k-1}$, and the quarantined vertices can safely be released.

		\subsection{Algorithmic Description}
		
		We continue by explaining the mechanism in our algorithm that bypasses the mixing of vertices queried during the construction of the first path, and the rest of the grid.

The idea is the following. We reserve an initial pool $V_{\text{start}}$ and
use it exclusively to construct the starting path. Once the path is complete,
every unused vertex of $V_{\text{start}}$ is moved to a set $V_q$ of
\emph{quarantined} vertices. We then build the first interior layer using only
vertices from $V_G\setminus (V_q\cup P_{2k-1})$.

Once this buffer layer is complete, the active frontier no longer contains
vertices of the initial path. At that point the quarantined vertices can be
released without reintroducing the one-edge exposure history accumulated during
Phase 1.

		We formalize this temporal restriction of the candidate pool as follows:

		\begin{enumerate}
			\item \textbf{Phase 0: Virtual Global Thinning.}
			Before the search begins, we fix a constant $p_0\in(0,1)$ such that
			$(1-p_0)^8\ge 1/2$. If the input density satisfies $p\le p_0$, the
			algorithm works directly with the original graph. If $p>p_0$, we define an
			operational subgraph $\widetilde G$ by independently retaining each potential
			edge of $G$ with probability $p_0/p$. This thinning is global in
			distribution but lazy in implementation: the thinning variable for an edge is
			revealed only when that edge is queried. The resulting operational graph has law
			$G(n,p_0)$, and all subsequent searches are performed in $\widetilde G$.
			Since $\widetilde G\subseteq G$, any grid found in $\widetilde G$ is also a
			grid in the original host graph. For notational simplicity, after this phase we
			write $G$ and $p$ for the operational graph and its edge density.
			\item \textbf{Phase 1: The Linear Partition.} Before the algorithm begins, we formally partition the vertex set $V(G)$ into two disjoint subsets: an initial pool $V_{\text{start}}$ of size $\lfloor n/3\rfloor$, and a pristine reservoir $V_{\text{main}}$. The algorithm constructs the initial $2k-1$ boundary path by querying candidates \textit{exclusively} from $V_{\text{start}}$.

		\item \textbf{Phase 2: The Buffer Layer.} Once $P_{2k-1}$ is successfully embedded, we define the \textbf{quarantined set} $V_q\subset V_{\text{start}}$ as all remaining unmapped vertices from the initial pool. These vertices may have been queried during Phase~1 and may therefore carry one-edge exposure history relative to	$P_{2k-1}$. They are strictly quarantined and temporarily removed from the search space. The algorithm executes the first layer of Corner Closures (the first $2k-3$ interior vertices) drawing \textit{exclusively} from the pristine reservoir $V_{\text{main}}$. Because no vertex in $V_{\text{main}}$ was queried during the Linear Extension, their probabilistic relationship with $P_{2k-1}$ is entirely independent, allowing Lemma \ref{lem:corner_closure} to apply directly.

			\item \textbf{Phase 3: The Release.} Once the first interior L-shape is complete, the original path $P_{2k-1}$ becomes fixed in the grid. The boundary subject to corner closures now consists entirely of vertices drawn from $V_{\text{main}}$. Because the quarantined vertices in $V_q$ were never queried against this newly formed boundary, their history relative to the active frontier is completely clean. The quarantine is formally lifted, and $V_q$ is merged back into the active reservoir for the remainder of the embedding, and Lemma \ref{lem:corner_closure} can be used.
			
		\end{enumerate}
		In Figure \ref{fig:grid_alg} we show the steps needed to construct a six-by-six grid. The same construction applies to larger square grids.

\begin{figure}[!htbp]
	\centering
	\includegraphics[width=0.6\linewidth]{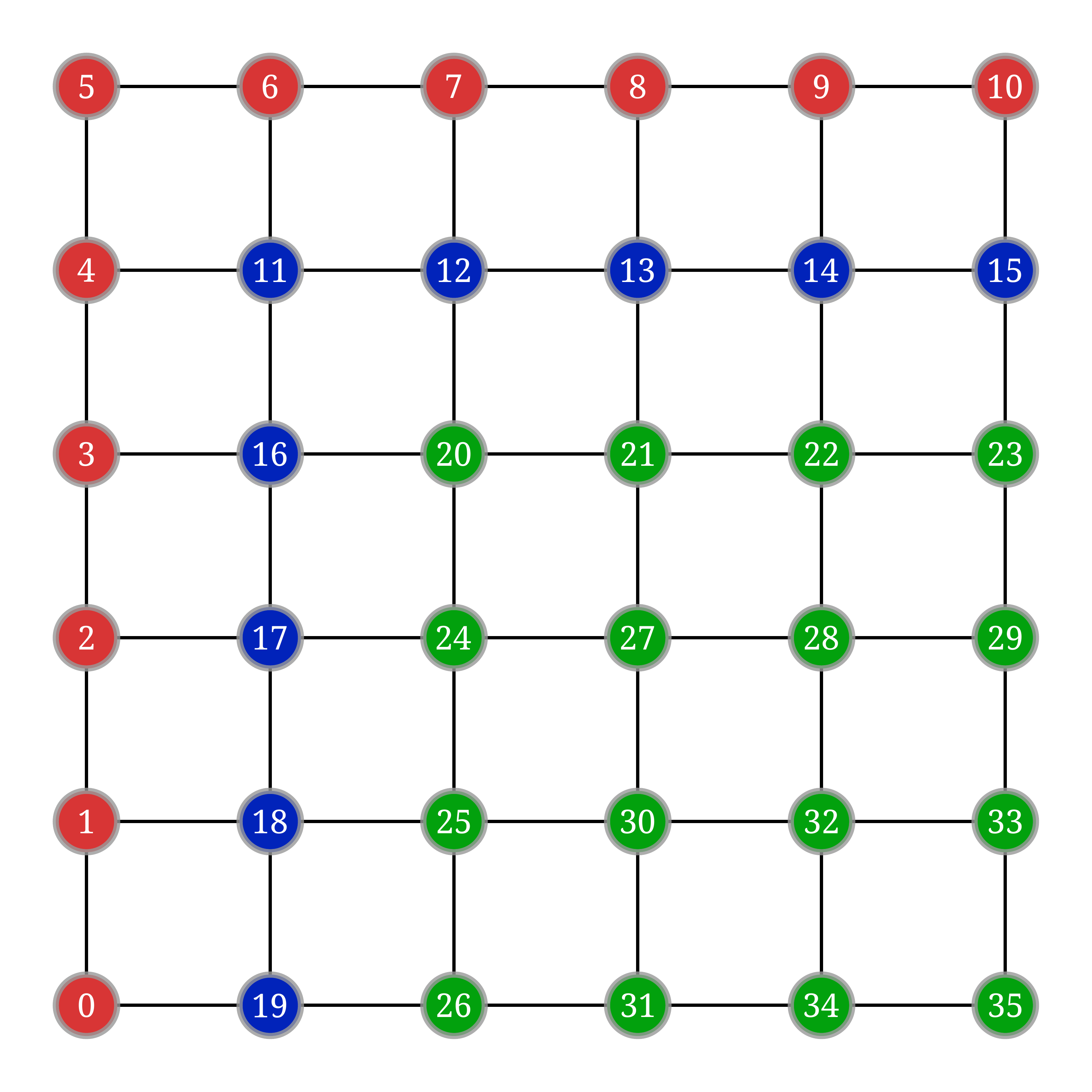}
\caption{Different colors indicate the phases in which vertices are embedded. Red vertices are added during Phase~1, blue vertices during Phase~2, which creates the quarantine buffer, and green vertices during Phase~3. Vertex labels indicate the chronological embedding order. Whenever a new layer is started, its corner is closed first and the	construction then proceeds outward toward the boundary.}
	\label{fig:grid_alg}
\end{figure}
		
		The formal description of our algorithm goes as follows:
		
		\clearpage
		\begin{algorithm}[!htbp]
			\footnotesize
			\caption{Quarantined Local Search ($\mathcal{A}$)}\label{alg:grid_search}
			
\textbf{Require:} Erd\H{o}s--Rényi graph $G=(V_G,E_G)$, target grid dimension $k$, model parameter $p$, and fixed thinning constant $p_0\in(0,1)$.
			
			\textbf{Ensure:} Embedded $k \times k$ square grid $H$, or \textbf{FAILURE}
			
			\vspace{1mm}
			
			\hrule
			
			\vspace{2mm}
			
			\begin{tabbing}
				
				\hspace{1.5em} \= \hspace{1.5em} \= \hspace{1.5em} \= \hspace{1.5em} \= \kill
				\textbf{Phase 0: Virtual Global Thinning} \\
				0. \> Fix $p_0\in(0,1)$ so that $(1-p_0)^8\ge 1/2$, and set $p_{\mathrm{op}}:=\min\{p,p_0\}$. \\
				1. \> \textbf{if} $p\le p_0$ \textbf{then} set $\widetilde G:=G$. \\
				2. \> \textbf{if} $p>p_0$ \textbf{then} assign independently to every potential edge \\
				\> $e\in {V_G\choose 2}$ a thinning variable $\xi_e\sim\mathrm{Bernoulli}(p_0/p)$. \\
				3. \> Define $e\in E(\widetilde G)$ iff $e\in E(G)$ and $\xi_e=1$. \\
				\> The variables $\xi_e$ are revealed only when the corresponding edge is queried. \\
				4. \> All subsequent searches and embeddings are performed in $\widetilde G$. \\
				5. \> For notational simplicity, relabel $\widetilde G$ as $G$ and $p_{\mathrm{op}}$ as $p$. \\
				\textbf{Phase 1: The Linear Partition} \\
				
				6. \> Partition $V_G$ into $V_{\text{start}}$ with $|V_{\text{start}}|=\lfloor n/3\rfloor$, and $V_{\text{main}}=V_G\setminus V_{\text{start}}$ \\
				
				7. \> Initialize path $P \leftarrow$ arbitrary root vertex $v \in V_{\text{start}}$ \\
				
				8. \> \textbf{while} $|P| < 2k-1$ \textbf{do} \\
				
				\> \> \textbf{Operation:} Linear Extension \\
				
				\> \> Search $V_{\text{start}} \setminus P$ for candidate $x$ adjacent to the current path tip \\
				
				\> \> \textbf{if} valid candidate found \textbf{then} permanently append $x$ to $P$ \\

				\> \> \textbf{else return FAILURE} \quad \textit{// Starvation} \\
				
				9. \> Map $P$ to the L-shaped outer boundary of target grid $H$ \\
				
				\\
				
				\textbf{Phase 2: The Buffer Layer} \\
				
				10. \> $V_q \leftarrow V_{\text{start}} \setminus P$ \quad \textit{// Quarantine remaining Phase 1 vertices} \\
				
				11. \> Reservoir $R \leftarrow V_{\text{main}}$ \quad \textit{// Set active reservoir to strictly pristine vertices} \\
				
				12. \> \textbf{while} the first interior layer ($2k-3$ vertices) is incomplete \textbf{do} \\
				
				\> \> \textbf{Operation:} Corner Closure \\
				
				\> \> Select the next closure pair $(u,v)$ according to the deterministic nested-$L$
				order of Lemma~\ref{lem:nested_L}. \\
				\> \> Search the active reservoir $R$ by pair-queries for a candidate $x$ adjacent to both $u$ and $v$. \\
				\> \> \textbf{if} valid candidate found \textbf{then} add $x$ to $H$ and remove from $R$ \\
				
				\> \> \textbf{else return FAILURE} \\
				
				\\
				
				\textbf{Phase 3: The Release} \\
				
				13. \> $R \leftarrow R \cup V_q$ \quad \textit{// Lift quarantine; merge initial vertices back into active reservoir} \\
				
				14. \> \textbf{while} $|V(H)| < k^2$ \textbf{do} \\
				
				\> \> \textbf{Operation:} Corner Closure \\
				
				\> \> Select the next closure pair $(u,v)$ according to the deterministic nested-$L$ order of Lemma~\ref{lem:nested_L}. \\
				\> \> Search the active reservoir $R$ by pair-queries for a candidate $x$ adjacent to both $u$ and $v$. \\
				\> \> \textit{*Note: On failure, the specific missing edge is deliberately ignored to preserve symmetry.} \\
				\> \> \textbf{if} valid candidate found \textbf{then} add $x$ to $H$ and remove from $R$ \\
				
				\> \> \textbf{else return FAILURE} \\
				
				\\
				
				15.\> \textbf{return} $H$
				
			\end{tabbing}
			
			\vspace{1mm}
			
			\hrule
			
		\end{algorithm}
		
The thinning in Phase 0 is global in distribution but lazy in implementation.
If $p>p_0$, then the independent variables $(\xi_e)_{e\in {V_G\choose 2}}$ define an operational graph $\widetilde G\sim G(n,p_0)$. If $p\le p_0$, then $\widetilde G=G$. The algorithm does not need to materialize $\widetilde G$; the thinning variables are revealed only when their corresponding edges are queried. In the probability estimates below, we write $G$ and $p$ for this operational graph and its edge density. In particular, throughout the proof we may assume $p\le p_0$, while success in the operational graph implies success in the original host graph. Since $p_0$ is a fixed positive constant and $C_\delta\sqrt{\ln k/n}=o(1)$, the lower bound $p\ge C_\delta\sqrt{\ln k/n}$ is preserved in the operational graph for all sufficiently large $n$.

We ignore integer-rounding issues throughout; replacing $n/3$ by
$\lfloor n/3\rfloor$ changes only lower-order terms and does not affect any
asymptotic estimate.

	We record the complexity of the construction.
	\begin{proposition}\label{prop:query_complexity}
		The Quarantined Local Search uses $O(k^2n)$ local candidate queries.
		Under the hypothesis $k^2\le n$, its query complexity is therefore
		$O(n^2)$.
	\end{proposition}
	
	\begin{proof}
		The construction performs $2k-2$ Linear Extensions and
		\begin{equation*}
			k^2-(2k-1)=(k-1)^2
		\end{equation*}
		Corner Closures. An exhaustive implementation inspects at most $n$
		candidates in each operation. Hence the total number of candidate
		queries is
		\begin{equation*}
			O(kn)+O(k^2n)=O(k^2n).
		\end{equation*}
		A Corner Closure requires at most two ordinary edge inspections per
		candidate, so implementing pair-queries at the edge level changes the
		bound only by a constant factor. Since $k^2\le n$, this is $O(n^2)$.
	\end{proof}

		\section{Proof of Theorem \ref{thm:main}}\label{sec:Proof}

In this section we prove the success of each phase using probabilistic bounds for the operational Erd\H{o}s--Rényi graph. We then prove the algorithm's correctness, and hence Theorem~\ref{thm:main}, by applying a union bound over the three phases of the algorithm.

		\subsection{Probability of Success for Phase 1}\label{subsec:proof:phase1}

We first prove that Phase 1 succeeds with high probability when the initial search is restricted to a linear-sized vertex pool $V_{\text{start}}$.

	\begin{lemma}\label{lem:phase1_success}
		Fix $\delta\in(0,1)$, and let $k=k(n)\to\infty$ satisfy
		\begin{equation*}
			k^2\le(1-\delta)n.
		\end{equation*}
		Suppose
		\begin{equation*}
			p\ge C_\delta\sqrt{\frac{\ln k}{n}},
		\end{equation*}
		where $C_\delta>0$ is fixed. If $F_1$ denotes failure of Phase~1,
		then
		\begin{equation*}
			\mathbb P(F_1)=o(1).
		\end{equation*}
	\end{lemma}
		
		\begin{proof}
			
			To prove the initial path is successfully embedded, we must show that the algorithm does not fail during any of the required extension steps. By Lemma \ref{lem:linear_ext}, the probability that the algorithm fails during a single linear extension step is bounded by $(1-p)^{N - 2k + 2}$, where $N$ is the total size of the available candidate pool.

			Because the search is confined to the initial partition, $N = n/3$. We apply the union bound over all $2k-2$ required extension steps to establish the global failure probability for Phase 1:
			
			\begin{align*}
				\mathbb{P}(\text{Phase 1 Fails}) \le (2k-2) \cdot (1-p)^{n/3 - 2k + 2} \le 2k \cdot \exp\left(-p\left(\frac{n}{3} - 2k + 2\right)\right).
			\end{align*}

			Since $k^2\le (1-\delta)n$, we have $2k=o(n)$, and hence, for all sufficiently large $n$,

			\begin{equation*}
				\frac{n}{3} - 2k + 2 > \frac{n}{3} - \frac{n}{6} = \frac{n}{6}.
			\end{equation*}

			We apply this to bound the negative exponent:
			
			\begin{equation*}
				-p \left(\frac{n}{3} - 2k + 2\right) < -p \left(\frac{n}{6}\right) \le -C_\delta\sqrt{\frac{\ln k}{n}} \cdot \frac{n}{6} = -\frac{C_\delta}{6}\sqrt{n \ln k}	
			\end{equation*}

			Substituting this explicit algebraic bound back into the global inequality yields:
			
			\begin{equation*}
				\mathbb{P}(\text{Phase 1 Fails}) < 2k \cdot \exp\left(-\frac{C_\delta}{6}\sqrt{n \ln k}\right)
				\end{equation*}

			Because the $\sqrt{n}$ term inside the exponential function grows faster than the $\ln(2k)$ required to cancel the polynomial multiplier, the negative exponential dominates. As $n \to \infty$, this upper bound vanishes to $0$, proving the algorithmic step for Phase 1 succeeds with high probability.
			
		\end{proof}
		
		\subsection{Probability of Success in Phase 2}
		
We continue by showing that, even after excluding the quarantined vertices, the remaining $2n/3$ vertices are sufficient to construct a buffer layer between the initial path and the rest of the grid.

	\begin{lemma}\label{lem:phase2_success}
		Fix $\delta\in(0,1)$, and let $k=k(n)\to\infty$ satisfy
		\begin{equation*}
			k^2\le(1-\delta)n.
		\end{equation*}
		Suppose
		\begin{equation*}
			p\ge C_\delta\sqrt{\frac{\ln k}{n}},
		\end{equation*}
		where $C_\delta>2$. If $F_2$ denotes failure of Phase~2 and
		$S_1$ denotes success of Phase~1, then
		\begin{equation*}
			\mathbb P(F_2\mid S_1)=o(1).
		\end{equation*}
	\end{lemma}
		
		\begin{proof}
			
    Condition throughout on the success event $S_1$. Since Phase~1 queries only edges with both endpoints in $V_{\text{start}}$, no edge involving a vertex of $V_{\text{main}}$ has been revealed. Thus the candidate reservoir used in Phase~2 is pristine at the beginning of the phase, and as Phase~2 progresses its candidates accumulate only the pair-query history covered by Lemma~\ref{lem:corner_closure}.

It therefore remains to show that, conditional on $S_1$, none of the
$2k-3$ required Corner Closures fails.
			
By Lemma~\ref{lem:corner_closure}, conditional on the revealed history,
the probability that any single Corner Closure fails is at most
\begin{equation*}
	\left(1-p^2(1-p)^8\right)^N,
\end{equation*}
where $N$ is the number of candidates remaining in the active reservoir.
The initial size of $V_{\text{main}}$ is $2n/3$, and during Phase~2 the
algorithm permanently embeds at most $2k-3$ vertices from this reservoir.
Hence throughout the phase,
\begin{equation*}
	N\ge \frac{2n}{3}-2k.
\end{equation*}

Applying the union bound over the $2k-3$ required Corner Closures gives
\begin{align*}
	\mathbb P(F_2\mid S_1)
	\le
	(2k-3)
	\left(1-p^2(1-p)^8\right)^{\frac{2n}{3}-2k} \le
	2k\exp\left(
	-p^2(1-p)^8
	\left(\frac{2n}{3}-2k\right)
	\right).
\end{align*}

 We substitute our operational density bound $p \ge C_\delta \sqrt{\frac{\ln k}{n}}$ into the exponent. For sufficiently large $n$, the available candidate pool $\frac{2n}{3} - 2k \ge \frac{n}{2}$. Substituting these bounds into the exponential penalty yields:
\begin{equation*}
	-p^2(1-p)^8 \left(\frac{2n}{3} - 2k\right) \le -C_\delta^2 \frac{\ln k}{n} (1-p)^8 \left(\frac{n}{2}\right) = -\frac{C_\delta^2}{2} \ln(k) (1-p)^8
\end{equation*}

Substituting this back into the global failure bound produces:
\begin{equation*}
	\mathbb P(F_2\mid S_1) \le 2k \cdot \exp\left(-\frac{C_\delta^2}{2} \ln(k) (1-p)^8\right) = 2 \cdot k^{1 - \frac{C_\delta^2}{2} (1-p)^8}
\end{equation*}

By Phase~0, the operational density satisfies $p\le p_0$, and hence
\begin{equation*}
	(1-p)^8\ge (1-p_0)^8\ge\frac12.
\end{equation*}
Therefore
\begin{equation*}
	\mathbb P(F_2\mid S_1)
	\le 2k^{1-C_\delta^2/4},
\end{equation*}
which tends to zero whenever $C_\delta>2$.
		\end{proof}

\subsection{Probability of Success in Phase 3}
Now that the first two phases have been proven we can move on to showing that Phase 3 will succeed with high probability as well.

\begin{lemma}\label{lem:phase3_success}
	Fix $\delta\in(0,1)$, and let $k=k(n)\to\infty$ satisfy
	\begin{equation*}
		k^2\le(1-\delta)n.
	\end{equation*}
	Suppose
	\begin{equation*}
		p\ge C_\delta\sqrt{\frac{\ln k}{n}},
	\end{equation*}
	where
	\begin{equation*}
		C_\delta>\frac{2}{\sqrt{\delta}}.
	\end{equation*}
	If $F_3$ denotes failure of Phase~3 and $S_1,S_2$ denote success
	of Phases~1 and~2, respectively, then
	\begin{equation*}
		\mathbb P(F_3\mid S_1\cap S_2)=o(1).
	\end{equation*}
\end{lemma}

\begin{proof}
Condition throughout on $S_1\cap S_2$. We must show that none of the final $(k-2)^2$ Corner Closures fails. At this stage the quarantine on $V_q$ has been lifted, so the active reservoir $R$ contains all remaining unmapped vertices of the graph.
	
	At the absolute final step of the algorithm, the grid contains $k^2 - 1$ vertices. Because the target grid is bounded by $k^2 \le (1-\delta)n$, the number of available candidates is bounded from below by $N \ge n - k^2 \ge \delta n$.
	
	More precisely, for every $x\in V_q$, each edge individually queried
	during Phase 1 has its other endpoint in the initial path
	$P_{2k-1}$. Once the first interior $L$-layer has been completed, no
	subsequent closure pair contains a vertex of $P_{2k-1}$. Hence the
	$x$-incident edge variables exposed during Phase 1 are disjoint from all
	$x$-incident edge variables that can occur in subsequent Corner Closure
	queries. Conditioning on the Phase-1 exposure history therefore does not
	alter the deferred-decisions argument of Lemma~\ref{lem:corner_closure}
	after the quarantine is lifted.

	By Lemma \ref{lem:corner_closure}, the sequential failure probability for any single step is bounded by $(1 - p^2(1-p)^8)^N$. Applying the union bound over all remaining operations, and absorbing the $(k-2)^2$ multiplier into the macroscopic upper bound $k^2$, we obtain the global failure probability:
	\begin{align*}
		\mathbb P(F_3\mid S_1\cap S_2) \le k^2 \left( 1 - p^2(1-p)^8 \right)^{\delta n}\le k^2 \cdot \exp\left( -p^2(1-p)^8(\delta n) \right)
	\end{align*}

We substitute our operational density bound $p \ge C_\delta \sqrt{\frac{\ln k}{n}}$ into the exponent:
\begin{equation*}
	-p^2(1-p)^8 (\delta n) \le -C_\delta^2 \frac{\ln k}{n} (1-p)^8 (\delta n) = -C_\delta^2 \delta \ln(k) (1-p)^8
\end{equation*}

Substituting this back into the global failure bound gives
\begin{equation*}
	\mathbb P(F_3\mid S_1\cap S_2)
	\le
	k^2 \exp\left(-C_\delta^2\delta\ln(k)(1-p)^8\right)
	=
	k^{2-C_\delta^2\delta(1-p)^8}.
\end{equation*}

By Phase~0, the operational density satisfies $p\le p_0$, and hence
\begin{equation*}
	(1-p)^8\ge (1-p_0)^8\ge \frac12.
\end{equation*}
Therefore,
\begin{equation*}
	\mathbb P(F_3\mid S_1\cap S_2)
	\le
	k^{2-C_\delta^2\delta/2}.
\end{equation*}
This upper bound tends to $0$ whenever
\begin{equation*}
	C_\delta>\frac{2}{\sqrt{\delta}}.
\end{equation*}
Thus Phase~3 succeeds with high probability.
\end{proof}

\subsection{Proof of Theorem \ref{thm:main}}

\begin{proof}
	Let $F_j$ and $S_j$ denote failure and success, respectively, of
	Phase $j$. The algorithm fails only if Phase 1 fails, Phase 2 fails after
	Phase 1 has succeeded, or Phase 3 fails after the first two phases have
	succeeded. Hence
	\begin{align*}
		\mathbb P(\text{Global Failure})=
		\mathbb P(F_1)
		&+\mathbb P(S_1\cap F_2)
		+\mathbb P(S_1\cap S_2\cap F_3)\le
		\mathbb P(F_1)\\
		&+\mathbb P(F_2\mid S_1)
		+\mathbb P(F_3\mid S_1\cap S_2).
	\end{align*}
	
	By the preceding lemmas, the first term tends to $0$, the second tends
	to $0$ whenever $C_\delta>2$, and the third tends to $0$ whenever
	$C_\delta>2/\sqrt{\delta}$. Therefore it is enough to choose
	\begin{equation*}
	C_\delta>
	\max\left\{2,\frac{2}{\sqrt{\delta}}\right\}.
	\end{equation*}
	For this choice the global failure probability tends to~$0$, and the
	algorithm embeds the $k\times k$ square grid with high probability.
\end{proof}

\section{Discussion and Open Problems}\label{sec:Discussion}
In this paper, we introduced a local algorithm for embedding large square grids in Erd\H{o}s--Rényi random graphs. The proof relies on two geometric features of the square grid. First, after an initial boundary path has been embedded, the remaining vertices can be exposed through a deterministic sequence of two-neighbour Corner Closures. Second, this closure schedule gives a uniform bound on the number of previous pair-tests involving any active boundary vertex. The quarantine step separates the one-edge information revealed while constructing the initial path from the later pair-query process. These features are precisely the points at which the proof uses the geometry of the square grid, and they explain why the argument does not immediately extend to other lattice families.

Previous constructive work of \citet{fernandez_manoussakis} already achieved the same $\sqrt{\log n/n}$ density order for square grids occupying a fixed linear fraction of the host graph. The present result extends the constructive guarantee to every fixed almost-spanning regime $k^2\le(1-\delta)n$. Meanwhile, the general embedding results of \citet{ferber_2017} apply to a substantially broader graph class but require higher densities when specialized to square grids. By exploiting the nested geometry of the square grid, the quarantine step separates the one-edge exposure created in Phase~1, while Lemma~\ref{lem:corner_closure} controls the accumulated pair-query history of candidates that remain in the reservoir. This permits a local polynomial-time construction in the almost-spanning regime.
 
The result can also be viewed as a constructive counterpart to the existential square-grid embedding results of \citet{riordan_2000}. Those results establish the presence of spanning structures at substantially sparser densities, whereas our argument produces an explicit local search procedure. In the regime $k^2=\Theta(n)$, our sufficient density is of order \begin{equation*}
\sqrt{\frac{\log n}{n}},
\end{equation*} 
which is a factor of order $\sqrt{\log n}$ above the corresponding $n^{-1/2}$ emergence scale. Thus the present result does not recover the existential threshold. Instead, it gives an explicit polynomial-time local construction at a density larger by a factor of order $\sqrt{\log n}$ in the regime $k^2=\Theta(n)$.

A natural quantitative question is whether this $\sqrt{\log n}$ overhead is intrinsic to the local-search framework. In the regime $k^2=\Theta(n)$, Theorem~\ref{thm:main} succeeds at density of order
\begin{equation*}
	\sqrt{\frac{\log n}{n}},
\end{equation*}
whereas the corresponding existential scale is $n^{-1/2}$. It is not clear whether the additional factor of order $\sqrt{\log n}$ reflects a genuine limitation of the present local-search framework, or only the failure bounds used in the present construction.

A second limitation comes from the reservoir requirement in Phase~3. The proof uses
\begin{equation*}
	n-k^2\ge \delta n,
\end{equation*}
so that a linear number of unused vertices remains available throughout the final Corner Closures. When $n-k^2=o(n)$, and in particular in the fully spanning case, this argument no longer gives the same starvation bound. Extending the method to this regime would therefore require either a more economical use of the reservoir or a different treatment of the final closure steps.

While our method gives a local construction for square grids, other
lattices pose different challenges. Our methodology is not directly applicable to hexagonal lattices.

Hexagonal walls illustrate a different obstruction. The natural layer-by-layer exposure of a wall does not admit the same structure in which, after a short initial boundary, essentially every remaining vertex is introduced through a two-neighbour Corner Closure. A substantial number of vertices must instead be introduced with only one previously embedded neighbour. Such steps behave like Linear Extensions rather than Corner Closures and therefore do not benefit from the $p^2$ success mechanism used throughout the interior of the square-grid construction.

One possible indirect approach would be to embed a suitably chosen rectangular grid that contains the desired wall as a subgraph. Such a construction would inherit the denser square-grid requirement, however, and would therefore not exploit the natural sparsity of the wall. A direct almost-spanning wall algorithm would require a different treatment of the many one-neighbour extension steps.

Triangular grids appear to be a natural direction for extending the method. The deferred-decisions argument in Lemma~\ref{lem:corner_closure} is not specific to two-neighbour closures, and a fixed three-neighbour closure should admit an analogous bound of the form $p^3(1-p)^C$. The main additional issue is that natural triangular-grid constructions may mix two- and three-neighbour closures: a failed query against a pair $(u,v)$ makes that candidate unusable for a later triple containing the same pair. Thus an extension would require additional control of the closure order or of the candidate reservoir. We leave this question open.

These examples show that the applicability of the Quarantined Local Search depends strongly on the geometry of the target graph. In the square-grid case, the initial boundary is constructed through Linear Extensions, and the remainder of the grid is completed through a deterministic sequence of Corner Closures for which the pair-test history of the active frontier remains uniformly bounded. The quarantine mechanism separates these two parts of the construction. A natural direction for future work is to determine which other graph families admit a similar construction.

\section*{Data Availability Statement}
Data sharing is not applicable to this article as no datasets were
generated or analyzed during the current study.
\section*{Conflict of Interest}
The authors declare no conflicts of interest.

\section*{Acknowledgements}
This work was supported in part by the Dutch Research Council (NWO) Talent Programme ENW-Vidi 2021 under grant number VI.Vidi.213.163 (DF).

       \bibliographystyle{plainnat}
		
		\bibliography{references}

	\end{document}